\documentclass[10pt]{amsart}
\usepackage{amsmath,amscd}
\usepackage{amsbsy}
\usepackage{amssymb}
\usepackage{amscd,amsthm}
\usepackage[all,cmtip]{xy}
\usepackage[russian,english]{babel}
\usepackage[utf8]{inputenc}
\usepackage{xcolor}

\newtheorem{thm}{Theorem}\numberwithin{thm}{section}
\newtheorem{lem}[thm]{Lemma}

\newtheorem{prob}[thm]{Problem}

\newcommand{\Mod}[1]{\, (\mathrm{mod}\ #1)}
\DeclareMathOperator{\li}{li}

\begin{document}
\begin{center}
\huge{Some Diophantine equations involving arithmetic functions and Bhargava factorials}\\[1cm] 
\end{center}
\begin{center}

\large{Daniel M. Baczkowski and Sa$\mathrm{\check{s}}$a Novakovi$\mathrm{\acute{c}}$}\\[0,5cm]
{\small June 2024}\\[0,5cm]
\end{center}
{\small \textbf{Abstract}. 
F. Luca proved for any fixed rational number $\alpha>0$ that the Diophantine equations of the form $\alpha\,m!=f(n!)$, where $f$ is either the Euler function or the divisor sum function or the function counting the number of divisors, have only finitely many integer solutions $(m,n)$. In this paper we generalize the mentioned result and show that Diophantine equations of the form $\alpha\,m_1!\cdots m_r!=f(n!)$ have finitely many integer solutions, too. In addition, we do so by including the case $f$ is the sum of $k$\textsuperscript{th} powers of divisors function. Moreover, we observe that the same holds by replacing some of the factorials with certain examples of Bhargava factorials. }

\begin{center}
\tableofcontents
\end{center}
\section{Introduction}
Diophantine equations involving factorials have a long and rich history. For example Brocard \cite{BR}, and independently Ramanujan \cite{RA}, asked to find all integer solutions for $n!=x^2-1$. Up to now this is an open problem, known as Brocard's problem. It is believed that the equation has only three solutions $(x,n)=(5,4), (11,5)$ and $(71,7)$. Overholt \cite{O} observed that a weak form of Szpiro's-conjecture implies that Brocard's equation has finitely many integer solutions. Using the ABC-conjecture Luca \cite{L} proved that Diophantine equations of the form $n!=P(x)$ with $P(x)\in \mathbb{Z}[x]$ of degree $d\geq 2$ have only finitely many integer solutions with $n>0$. If $P(x)$ is irreducible, Berend and Harmse \cite{BH1} showed unconditionally that $P(x)=H_n$ has finitely many integer solutions where $H_n$ are highly divisible sequences which also include $n!$. Furthermore, they proved that the same is true for certain reducible polynomials. 

Without assuming the ABC-conjecture, Berend and Osgood \cite{BO} showed that for arbitrary $P(x)$ of degree $\geq2$ the density of the set of positive integers $n$ for which there exists an integer $x$ such that $P(x)=n!$ is zero. 
Further progress in this direction is obtained by Bui, Pratt and Zaharescu \cite{BPZ} where the authors give an upper bound on integer solutions $n\leq N$ to $n!=P(x)$. 
Of course, there are several polynomials $P(x)$ for which $P(x)=n!$ is known to have either very few integer solutions or none (see for instance \cite{EO}, \cite{DAB} and \cite{PS}). 
Berndt and Galway \cite{BG} showed that the largest value of $n$ in the range $n<10^9$ for which Brocard's equation $x^2-1=n!$ has a positive integer solution is $n=7$. Matson \cite{MA} extended the range to $n<10^{12}$ and Epstein and Glickman \cite{EG} to $n<10^{15}$. 

Starting from Brocard's problem there have also been studied variations or generalizations of $x^2-1=n!$ (see for instance \cite{DMU}, \cite{KL}, \cite{MU}, \cite{MUT} \cite{TY}). For instance Ulas \cite{MU} studied, among others, Diophantine equations of the form $2^nn!=y^2$ and proved that the Hall conjecture (which is a special case of ABC-conjecture) implies that the equation has only finitely many integer solutions. Note that $2^nn!$ can also be formulated using the notation for the Bhargava factorial $n!_S$. Indeed, it is  $2^nn!=n!_S$, with $S=\{2x+b| x\in \mathbb{Z}\}$. We do not recall the definition of the Bhargava factorial and refer to \cite{BH} or \cite{WT} instead. Other Diophantine equations involving factorials have proved more tractable. For example, Erd\"os and Obl\'ath \cite{EO} proved that $n!\pm m!=x^p$ has only finitely many integer solutions. This result has been generalized to equations of the form $An!+Bm!=P(x)$ where $P(x)$ is a polynomial with rational coefficients of degree $\geq 2$. Luca \cite{L2} proved that if the degree is $\geq 3$, then the ABC-conjecture implies that $An!+Bm!=P(x)$ has finitely many integer solutions, except when $A+B=0$. In this case there are only finitely many solutions with $n\neq m$. For polynomials of the form $P(x)=a_dx^d+a_{d-3}x^{d-3}+\cdots +a_0$ with $a_d\neq 0$ and $d\geq 2$ Gawron \cite{MG} showed that there are finitely many integer solutions to $n!+m!=P(x)$, provided the ABC-conjecture holds. In \cite{NO} and \cite{NOO} the second author studied Diophantine equations of the form $am_1!\cdots m_r!=f(x,y)$ and $An!+Bm!=f(x,y)$, generalizing in this way results of Takeda \cite{WT} and Luca \cite{L2}. More concretely, the second author proved that under certain conditions these equations have finitely many integer solutions as well. The reader interested in results and open problems concerning Diophantine equations involving factorias  should consult Guy's excellent book \cite{GUY}. 

In the present note we are interested in Diophantine equations that are inspired by a paper of Luca \cite{LUC}. To be precise, in \emph{loc. cit.} Luca studied equations of the form 
\begin{center}
	$\alpha\,m!=f(n!)$,
\end{center}
where $\alpha$ is a positive rational number and $f$ either the Euler function $\phi$ or the divisor sum function $\sigma$ or the function counting the number of divisors $\sigma_0$. The main result in \cite{LUC} is the following: 
\begin{thm}[F. Luca]\label{lucaf}
	Let $\alpha$ be a positive rational number and $f$ either the Euler function or the divisor sum function or the function counting the number of divisors. Then the equation
	\begin{eqnarray*}
		\alpha\,m!=f(n!)
		\end{eqnarray*} 
has only finitely many solutions.
\end{thm}
Moreover, Luca determined all the solutions for $\alpha=1$. In \cite{BAC} the first author, Filaseta, Luca and Trifonov generalized the results obtained in \cite{LUC} by considering Diophantine equations of the form $bf(n!)=am!$. In particular, it is shown that under certain conditions there are at most finitely many positive integer solutions $(b,n,a,m)$. In view of these results it is natural to study Diophantine equations of the form 
\begin{eqnarray}\label{one}
	\alpha\,m_1!\cdots m_r!=f(n!)
\end{eqnarray} 
where $\alpha$ is a positive rational number and $f$ denotes the functions from above. 
We prove that (\ref{one}) has finitely many solutions for the same functions $f$ and also $f=\sigma_k$ for every integer $k\ge 1$. Here for a positive integer $N$, we use $\sigma_k(N)$ to denote the sum of the $k$\textsuperscript{th} power of positive divisors of $N$ (e.g. $\sigma=\sigma_1$).  Furthermore, we prove a more general result with some of the factorials replaced by more general Bhargava factorials.  The main result of this paper is the following theorem.


\begin{thm}\label{main}
Let $\alpha$ be a positive rational number and fix a natural number $r$. Let $f\in\{\phi,\sigma_k\}$ with $k\ge 0$ an integer.  Fix integers $s_i \ge 1$ and $t_i$ for each $1\le i\le r$, and let $S_i=\{s_ix+t_i|x\in\mathbb{Z}\}$.  Let $T=\{x^2|x\in\mathbb{Z}\}$.  Then the Diophantine equations
	\begin{align*}
		\alpha\, m_1!\cdots m_r! &= f(n!) \\
		\alpha\, m_1!_{S_1}\cdots m_r!_{S_r} &= f(n!) \\
		\alpha\, m_1!_{T}\cdots m_r!_{T} &= f(n!) 
	\end{align*}
have only finitely many solutions $(n,m_1,...,m_r)$.
\end{thm}

\noindent We pose the related general problem below.
\begin{prob}
	\textnormal{Let $S,S_i\subset \mathbb{Z}$ be infinite subsets and $\alpha>0$ a rational number. Investigate the set of solutions $(n,m_1,...,m_r)$ of}
	\begin{center}
	$\alpha\prod_{i=1}^r m_i!_{S_i}=f(n!_S)$.
\end{center}
\end{prob}

\section{Preliminaries}
Throughout this work we let $p$ and $q$ denote primes.  We let $\log_q n = (\log n)/(\log q)$ denote the base $q$ logarithm.  For a prime $q$ and integer $N \ge 2$, we let $\nu_q(N)$ denote the  exponent of $q$ in the prime factorization of $N$.  We use the convention that $\nu_q(1)=0$. For a rational number $a/b$, we let $\nu_q(a/b) = \nu_q(a) - \nu_q(b)$.  When we say a prime $q$ does not divide a rational number $\alpha$, we mean $\nu_q(\alpha)=0$.

We use Vinogradov notation $\ll$ and $\gg$.  By $f\ll g$, we mean $f=O(g)$ (i.e., there exists a constant $C$ such that $f \le C g$), and by $f\gg g$, we mean $g=O(f)$.  We use $f\approx g$ to mean $f\ll g$ and $f\gg g$.  For example, $\nu_q(N!) \approx N/q$ with $q$ fixed.  We use $f\sim g$ to denote $f$ is asymptotic to $g$.  

It is known that $\nu_p(N!) = (N-s_p(N))/(p-1)$ where $s_p(N)$ denotes the sum of the base-$p$ digits of $N$ (cf.~\cite{bach}).  Since $1 \le s_p(n) \le (p-1) \lfloor \log_p n\rfloor$ for $p\le n$ and $\nu_p(N!)=0$ for $p>N$, we have the following result that we will use repeatedly.

\begin{lem}\label{n!exp}
Let $N$ be a positive integer.  For any prime $p\le N$,
\[
\nu_p(N!) = \begin{cases}
 \frac{N}{p-1} + O( \log_p N ) & \text{if}~p\le N \\
 \frac{N}{p-1} + O( 1 ) & \text{if}~p> N \\
 \end{cases}
\]
\end{lem}

Furthermore, let $I_\ell = (n/(\ell+1),n/\ell]$ for any $\ell \ge 1$.  Note for $p>\sqrt{n}$ that $\nu_p(n!) = \lfloor n/p \rfloor$.  Thus, we deduce the following results, which will be used frequently within our work.

\begin{lem}\label{nup}
Let $n,\ell\ge 1$ be positive integers.  Let $I_\ell = (n/(\ell+1),n/\ell]$.  If $\ell+1 \le \sqrt{n}$, then $\nu_p(n!) = \ell$ for every $p\in I_\ell$. 
\end{lem}

The following lemma is a consequence of the Siegel-Walfisz theorem~(e.g., see Corollary 11.21 of \cite{mv}).  It is used to prove Lemma~\ref{nuqQ}. 

\begin{lem}\label{siegel}
Let $A>0$ be given.  There exists a constant $c>0$ such that if $h\le (\log x)^A$ and $\gcd(a,h)=1$, then
\[
\pi(x;h,a) \,:= \!\! \sum_{p\le x\atop p\equiv a\Mod{h}} \!\!1 ~=~ \frac{\li(x)}{\phi(h)} + O\Big( \frac{x}{\exp(c\sqrt{\log x} )} \Big)
\]
where $\li(x) = \int_2^x \frac{du}{\log u}$ denotes the logarithmic integral.
\end{lem}

We will also make use of the next result, which provides an estimate for the exponents in the prime factorization of the product $\prod_{p\le n} (p-1)$.

\begin{lem}\label{nuqQ}
Let $A>0$ be given.  For any prime $q \le (\log n)^A$, we have
\[
\nu_q\Big( \prod_{p\le n} (p-1) \Big) =  \frac{q}{(q-1)^2}  \li(n) + O\Big( \frac{n}{(\log n)^A} \Big).
\]
\end{lem}
\begin{proof}
Let $A>0$ be given, and fix a prime $q \le (\log n)^A$.  Let $J=\lfloor A\log_q( \log n) \rfloor$, and notice that $J\ge 1$.  For some prime $p\le n$, observe that $q^j \mid (p-1)$ if and only if $p\equiv 1\Mod{q^j}$.  Moreover, as $p-1<n$, note that $1\le j\le \log_q n$.  We proceed by counting the number of primes $p\le n$ such that $p\equiv 1\Mod{q^j}$ for each such $j$.  Thus,
\[
\nu_q\Big( \prod_{p\le n} (p-1) \Big) = \sum_{1\le j\le J} \pi(n;q^j,1) ~+ \sum_{J+1\le j\le \log_q n} \!\!\! \pi(n;q^j,1).
\]
Let $S_1$ and $S_2$ denote the latter two sums, respectively.  For each $1\le j\le J$, take note that $q^j\le q^J \le (\log n)^A$.  Also, note that
$J+1 > A\log_q(\log n)$, which implies
\[
\frac{q}{q^J (q-1)^2} = \frac{q^2}{q^{J+1}(q-1)^2} \ll \frac{1}{q^{J+1}} < \frac{1}{(\log n)^A}.
\] 
Recall the well known fact that $\li(n) \sim n/\log n$.  By Lemma~\ref{siegel}, we have
\begin{align*}
S_1 &= \sum_{1\le j\le J} \frac{\li(n)}{\phi(q^j)} + O\Big( \frac{J n}{e^{c\sqrt{\log n}}} \Big) \\
&=  \frac{q-\frac{q}{q^J}}{(q-1)^2}  \li(n) + O\Big( \frac{n\log n}{e^{c\sqrt{\log n}}} \Big) \\
&=  \frac{q}{(q-1)^2}  \li(n) + O\Big( \frac{n}{(\log n)^{A+1}} \Big)
\end{align*}
where the latter equality holds as $\li(n) \ll n/\log n$ and the note above about $\frac{q}{q^J(q-1)^2}$ imply their product gives the new error term, and the previous error term gets absorbed into the new error term.  Next, since $\pi(n;q^j,1)$ is at most the number of integers $\le n$ divisible by $q^j$, we deduce that $\pi(n;q^j,1) \le \lfloor \frac{n}{q^j} \rfloor$.  Thus, 
\begin{alignat*}{2}
S_2 &\le \sum_{J+1\le j\le \log_q n} \left\lfloor \frac{n}{q^j} \right\rfloor &&\le \sum_{J+1\le j\le \log_q n}\left( \frac{n}{q^j} +1 \right) \\
&< \sum_{j\ge J+1} \frac{n}{q^j} + \log_q n &&\ll \frac{n}{(\log n)^A} 
\end{alignat*}
where the latter equality holds again as $q^{J+1} > (\log n)^A$.  The result follows upon putting together the estimates for $S_1$ and $S_2$.
\end{proof}

Before proceeding to the proof of Theorem~\ref{main}, we note that $m_i!_{S_i} = s_i^{m_i} m_i!$ for each $1\le i\le r$ and $m!_T=(1/2)(2m)!$.

\section{Proof of Theorem~\ref{main} for Euler $\phi$ function}

Assume for large $n$ that a solution exists to $\alpha\,m_1!_{S_1}\cdots m_r!_{S_r}=\phi(n!)$.  Note the case where $s_1=\cdots=s_r=1$ gives the first equation in the statement of the theorem. Without loss of generality, we assume $m_1\le \cdots \le m_r$.  Since $\phi(n!) = n! \prod_{p\le n} \frac{p-1}{p}$, notice that
\[
\frac{n!}{m_1!\cdots m_r!} = \alpha\,s_1^{m_1}\cdots s_r^{m_r} \prod_{p\le n} \frac{p}{p-1} 
\]
Let $\beta=\alpha\,s_1^{m_1}\cdots s_r^{m_r}$.  
Let $L$ denote the fraction on the left-hand side and $R$ denote the fraction on the right-hand side.  
Let $\mathcal{Q} = \prod_{p\le n} (p-1)$ and $d=n-\sum_{1\le i\le r} m_i$. 
By Lemma~\ref{n!exp} we have for any prime $q\le n$ that
\[
\nu_q(L) = \frac{d}{q-1} + O(r\log_q n) + O(r).
\]
Note the second error term can be absorbed into the first error term.  On the other hand, we have
\[
\nu_q(R) = \nu_q(\beta) + 1 - \nu_q(\mathcal{Q}).
\]
Thus, for any prime $q\le n$, we have
\[
\nu_q(\mathcal{Q}) = \frac{-d}{q-1} + \nu_q(\beta) + 1 + O(r\log_q n).
\]
Note $\beta$ is divisible by a fixed number of distinct primes, namely the distinct prime divisors of $\alpha\,s_1\cdots s_r$.  Let $q$ and $q'$ with $q\neq q'$ be the first two primes not dividing~$\beta$ (i.e., $\nu_q(\beta)=\nu_{q'}(\beta)=0$).  Since $\nu_q(\beta)=\nu_{q'}(\beta)=0$ and $q,q',r$ are all fixed, we deduce that 
\[
|(q-1)\nu_q(\mathcal{Q}) - (q'-1)\nu_{q'}(\mathcal{Q}) |= O(\log n).
\]
By Lemma~\ref{nuqQ} with say $A=2$, we have that 
\[
\frac{|q'-q|}{(q-1)(q'-1)} \li(n) = O\Big( \frac{n}{(\log n)^2} \Big)
\]
implying as $q$ and $q'$ are fixed that $\li(n) \ll \frac{n}{(\log n)^2}$.  Yet, it is well known that $\li(n)\sim \frac{n}{\log n}$.  Hence, we have $n$ must be bounded.  It follows that $m_1\le\cdots \le m_r$ are also bounded.  In the $f=\phi$ case, this proves the theorem for the first two equations.  

For the last equation, observe that $\alpha\,m_1!_T\cdots m_r!_T=\phi(n!)$ is equivalent to $\beta\,(2m_1)!\cdots (2m_r)!=\phi(n!)$ where $\beta=\alpha/2^r$.  Since $\beta$ is fixed and we established that the first equation with $f=\phi$ has a finite number of solutions, this equation also has a finite number of solutions.

\section{Proof of Theorem~\ref{main} for number of divisors function $\sigma_0$}
Assume for large $n$ that a solution exists to $\alpha\,m_1!_{S_1}\cdots m_r!_{S_r}=\sigma_0(n!)$, or equivalently, $\beta\,m_1!\cdots m_r! = \sigma_0(n!)$ with $\beta=\alpha\,s_1^{m_1}\cdots s_r^{m_r}$.  Note the case where $s_1=\cdots=s_r=1$ gives the first equation in the statement of the theorem. Without loss of generality, we assume $m_1\le \cdots \le m_r$. 

Note $\beta$ is divisible by a fixed number of distinct primes, namely the distinct prime divisors of $\alpha\,s_1\cdots s_r$.  Let $q$ be the first prime that does not divide $\beta$.  For $p\in I_{q-1}$, note that $\nu_p(n!) = \lfloor n/p \rfloor = q-1$, or equivalently $q=\nu_p(n!)+1$.  Thus, the exponent of $q$ in the prime factorization of $\sigma_0(n!)$ is at least the number of primes in $I_{q-1}$; that is, 
\[
\nu_q\big( \sigma_0(n!) \big) \ge \pi\Big( \frac{n}{q-1} \Big) - \pi\Big( \frac{n}{q} \Big).
\]
Using the logarithmic integral approximation of $\pi(x)$ in the prime number theorem with an appropriate error term, we have that the right-hand side above is $\sim \frac{n}{q(q-1)\log n}$.  Hence,
\[
\nu_q\big( \sigma_0(n!) \big) \ge \frac{n}{2q(q-1)\log n}.
\]
On the other hand, as $q\nmid \beta$,
\[
\nu_q( \beta\,m_1!\cdots m_r! ) = \sum_{1\le i\le r} \nu_q( m_i! ) < \sum_{1\le i\le r} \frac{m_i}{q-1} \le r \frac{m_r}{q-1}.
\]
Since $q$ and $r$ are fixed, we have that $\frac{n}{\log n} < 2qr m_r \ll m_r$. 

Now, $\log \beta = \log \alpha + \sum_{1\le i\le r} m_i\log s_i \ll r\, m_r \approx m_r$, since $\alpha,r,s_1,\ldots,s_r$ are all fixed. In addition, $m_1!\cdots m_r! \ge m_r!$.  Thus, 
\begin{align*}
n \ll m_r \log m_r \sim \log (\beta\,m_r!) &\le \log( \beta\,m_1!\cdots m_r!) \\
&= \log\big( \sigma_0(n!) \big) = \sum_{p\le n} \log \big( \nu_p(n!)+1 \big).
\end{align*}
Let $L=\log n$.  Next, we analyze the latter sum in two parts.  First, as $\nu_p(n!) < n$ for all primes $p$,
\[
\sum_{p\le n/L} \log \big( \nu_p(n!)+1 \big) \le \sum_{p\le n/L} \log n \ll \frac{n}{L\log (n/L)}\log n \ll \frac{n}{\log n}.
\]

Next, we use a version of the Brun-Titchmarsh inequality~\cite{ls} about primes in short intervals that says $\pi(x+y)-\pi(x) \le 2y/\log y$.  If we let $y$ denote the length of the interval $I_\ell$, then the number of primes in $I_\ell$ is $\le \frac{2y}{\log y}$.  Using that $\ell < L =\log n$, we have that the number of primes in $I_\ell$ is $\ll \frac{n}{\ell^2 \log n}$.  Now, for the primes $p\in (n/L,n]$, we use that $(n/L,n]$ can be partitioned into the disjoint union of $I_\ell$ with $1\le \ell < L$.  
For each $p\in I_\ell$, recall that $\nu_p(n!) = \ell$.  Thus,
\[
\sum_{n/L<p\le n} \log \big( \nu_p(n!)+1 \big) = \sum_{1\le \ell <L} \sum_{p\in I_\ell} \log (\ell+1) \ll \frac{n}{\log n} \sum_{\ell \ge 1} \frac{\log(\ell + 1)}{\ell^2} \approx \frac{n}{\log n}. 
\]
Putting together the latter two inequalities, we deduce that $n \ll \frac{n}{\log n}$.  Hence, we have that $n$ must be bounded.  It follows that $m_1\le\cdots \le m_r$ are also bounded.  In the $f=\sigma_0$ case, this proves the theorem for the first two equations.  

For the last equation, observe that $\alpha\,m_1!_T\cdots m_r!_T=\sigma_0(n!)$ is equivalent to $\beta\,(2m_1)!\cdots (2m_r)!=\sigma_0(n!)$ where $\beta=\alpha/2^r$.  Since $\beta$ is fixed and we established that the first equation with $f=\sigma_0$ has a finite number of solutions, this equation also has a finite number of solutions.

\section{Sum of $k$\textsuperscript{th} powers of divisor functions $\sigma_k$ with $k\ge 1$}

Fix an integer $k\ge 1$.  Again, assume for large $n$ that a solution exists to $\alpha\,m_1!_{S_1}\cdots m_r!_{S_r}=\sigma_k(n!)$, or equivalently, $\beta\,m_1!\cdots m_r! = \sigma_k(n!)$ with again $\beta=\alpha\,s_1^{m_1}\cdots s_r^{m_r}$.  Note the case where $s_1=\cdots=s_r=1$ gives the first equation in the statement of the theorem. Without loss of generality, we assume $m_1\le \cdots \le m_r$. 

Here we apply a result of Stewart~\cite{stew} that implies the greatest prime factor of $2^N-1$ is 
\[
>N\exp\Big\{ \frac{\log N}{104 \log\log N} \Big\}
\]
for sufficiently large $N$.  Let $N_2=\nu_2(n!)+1$.  Take note that $2^{kN_2}-1$ is a divisor of $\sigma_k(n!)$. Note $\beta$ is divisible by a fixed number of distinct primes, namely the distinct prime divisors of $\alpha\,s_1\cdots s_r$.  Since the largest prime divisor of $\beta\,m_1!\cdots m_r!$ is at most $m_r$ and $N_2 \ge n/2$, we have
\[
m_r > \frac{kn}{2}\exp\Big\{ \frac{\log (kn/2)}{104 \log\log (kn/2)} \Big\}.
\]
Note this implies that $\log m_r \gg \log n$.  From $\log N! \sim N\log N$, we find that
\[
\log m_r! \sim m_r \log m_r \gg n (\log n) \exp\Big\{ \frac{\log (kn/2)}{104 \log\log (kn/2)} \Big\}.
\]
Now, $\log \beta = \log \alpha + \sum_{1\le i\le r} m_i\log s_i \ll r\, m_r \approx m_r$, since $\alpha,r,s_1,\ldots,s_r$ are all fixed. In addition, $m_1!\cdots m_r! \ge m_r!$.  Thus, $\log (\beta m_r!) \sim m_r\log m_r \sim \log m_r!$.  
Note $\sigma_k(N) \le \sigma_1(N)^k$. From the well known fact that $\sigma_1(x)\ll x\log\log x$ and again that $\log N! \sim N\log N$, we have that $\log \sigma_1(n!) \sim \log n! \sim n\log n$.  Since $k$ is fixed,  $\log\sigma_k(n!) \le k\log\sigma_1(n!) \ll n\log n$. Thus, 
\begin{align*}
n (\log n) \exp\Big\{ \frac{\log (kn/2)}{104 \log\log (kn/2)} \Big\} &\ll \log m_r! \sim \log (\beta m_r!) \\
&\le \log(\beta m_1!\cdots m_r!) = \log \sigma_k(n!) \ll n \log n
\end{align*}
which implies that $n$ is bounded.  It follows that $m_1\le\cdots \le m_r$ are also bounded.  In the $f=\sigma_k$ with $k\ge 1$ case, this proves the theorem for the first two equations.  

For the last equation, observe that $\alpha\,m_1!_T\cdots m_r!_T=\sigma_k(n!)$ is equivalent to $\beta\,(2m_1)!\cdots (2m_r)!=\sigma_k(n!)$ where $\beta=\alpha/2^r$.  Since $\beta$ is fixed and we established that the first equation with $f=\sigma_k$ has a finite number of solutions, this equation also has a finite number of solutions.

\vspace{0.5cm}
\noindent
{\tiny DEPARTMENT OF MATHEMATICS, UNIVERSITY OF FINDLAY, OHIO 45840, USA.}\\
E-mail adress: baczkowski@findlay.edu\\
{\tiny HOCHSCHULE FRESENIUS UNIVERSITY OF APPLIED SCIENCES 40476 D\"USSELDORF, GERMANY.}\\
E-mail adress: sasa.novakovic@hs-fresenius.de\\
\noindent
{\tiny MATHEMATISCHES INSTITUT, HEINRICH--HEINE--UNIVERSIT\"AT 40225 D\"USSELDORF, GERMANY.}\\
E-mail adress: novakovic@math.uni-duesseldorf.de

\end{document}